\documentclass[12pt]{article}
\usepackage{url}
\usepackage{hyperref}
\usepackage[final]{optional}
\usepackage[a4paper, margin=2cm]{geometry}
\usepackage{enumerate}
\usepackage{graphics}
\usepackage{tikz}
\usepackage{tikz-3dplot} 
\usepackage{verbatim}
\usepackage{slashed}
\usepackage{hyperref}

\usepackage{datetime}
\usepackage{color}
\usepackage{amsmath,amsfonts,amsthm,amssymb}
\usepackage{epstopdf}
\usepackage[all,cmtip]{xy}
\usepackage{accents}
\usepackage[nameinlink,capitalise]{cleveref} % References show what kind of thing is being referenced (theorem, lemma, proposition...)
\usepackage{dashbox}

\newcounter{results}[section] % Uniform counters for lemmas, theorems, propositions etc

\newcounter{steps}[section] % Uniform counters for lemmas, theorems, propositions etc

\theoremstyle{plain}
\newtheorem{theorem}[results]{Theorem}
\newtheorem{remark}[results]{Remark}

\newtheorem{lemma}[results]{Lemma}

\newtheorem{corollary}[results]{Corollary}
\newtheorem{conj}[results]{Conjecture}

\newtheorem*{theorem*}{Theorem}
\newtheorem*{lemma*}{Lemma}
\newtheorem*{proposition*}{Proposition}
\newtheorem*{corollary*}{Corollary}
\newtheorem*{exercise*}{Exercise}
\newtheorem*{fact*}{Fact}

\newtheorem*{remark*}{Remark}
\newtheorem*{question*}{Question}

\theoremstyle{definition}

\newtheorem*{definition*}{Definition}
\newtheorem*{example*}{Example}

\numberwithin{equation}{section}

\newcommand{\R}{\mathbb{R}}

\newcommand{\Z}{\mathbb{Z}}
\newcommand{\Sp}{\mathbb{S}}

\newcommand{\mbb}{\mathbb}
\newcommand{\C}{\mbb{C}}

\newcommand{\lan}{\langle}
\newcommand{\la}{\lambda}
\newcommand{\ra}{\rangle}

\newcommand{\h}{\mathcal{H}}
\newcommand{\A}{\mathcal{A}}

\newcommand{\emb}{\hookrightarrow}

\newcommand{\In}{\subset}
\newcommand{\Om}{\Omega}
\newcommand{\om}{\omega}
\newcommand{\dl}{{\delta}}
\newcommand{\Dl}{{\Delta}}

\newcommand{\al}{{\alpha}}

\newcommand{\ed}{{\rm d}}
\newcommand{\id}{\,\,\ed}

\newcommand{\D}{{\nabla}}
\newcommand{\Db}{{\nabla^{\bot}}}

\newcommand{\eps}{{\varepsilon}}
\newcommand{\fr}[2]{\frac{#1}{#2}}
\newcommand{\sm}{{\setminus}}

\newcommand{\T}{{\rm T}}

\newcommand{\pl}[2]{{\frac{\partial #1}{\partial #2}}}

\newcommand{\zb}{\overline{z}}
\newcommand{\db}{\overline{\partial}}
\newcommand{\de}{\partial}
\newcommand{\dz}{{\de_z}}
\newcommand{\dzb}{{\de_{\zb}}}
\newcommand{\nn}{\nonumber}

\newcommand{\Si}{\Sigma}

\newcommand{\twopartdef}[4]
{
	\left\{
		\begin{array}{ll}
			#1 & \mbox{if  } #2 \bigskip \\
			#3 & \mbox{if  } #4
		\end{array}
	\right.
}
\newcommand{\threepartdef}[6]
{
	\left\{
		\begin{array}{lll}
			#1 & \mbox{if  } #2\bigskip \\
			#3 & \mbox{if  } #4 \bigskip \\
			#5 & \mbox{if  } #6
		\end{array}
	\right.
}

\def\XXint#1#2#3{{\setbox0=\hbox{$#1{#2#3}{\int}$}
     \vcenter{\hbox{$#2#3$}}\kern-.5\wd0}}

\newcommand{\ov}{\overline}
\DeclareMathOperator{\tr}{\rm tr}

\DeclareMathOperator{\Rm}{Rm}

\DeclareMathOperator{\Area}{Area}

\newcommand{\vlinesub}[1]{\vline_{_{_{_{_{_{_{_{\,#1}}}}}}}}}

\newcommand{\abs}[1]{\vert#1\vert}
\newcommand{\calj}{\mathcal{J}}

\newcommand{\secondvar}[3]{\mathrm{d}^2{#1}\rbrak{#2}\sbrak{#3}}
\newcommand{\cbrak}[1]{\left\{{#1}\right\}}
\newcommand{\rbrak}[1]{\left({#1}\right)}
\newcommand{\sbrak}[1]{\left[{#1}\right]}
\newcommand{\inner}[1]{\left\langle #1 \right\rangle}

\newcommand{\into}{\hookrightarrow}

\newcommand{\indx}{\ensuremath{\mathrm{index}}}
\newcommand{\nullity}{\ensuremath{\mathrm{nullity}}}

\newcommand{\via}{\textit{via}}

\newcommand{\pd}{\partial}
\newcommand{\dform}{\mathrm{d}}
\newcommand{\sff}{\mathrm{II}}
\newcommand{\tensor}{\otimes}

\allowdisplaybreaks

\begin{document}
\title{Index estimates for constant mean curvature surfaces in three-manifolds by energy comparison}
\author{Luca Seemungal and Ben Sharp}
\date{\today}
\maketitle

\begin{abstract}
We prove a linear upper bound on the Morse index of closed constant mean curvature (CMC) surfaces in orientable three-manifolds in terms of genus, number of branch points and a Willmore-type energy.    
\end{abstract}

%-----------INTRODUCTION-------------------------------

\section{Introduction}

Given $m\geq 3$ and a domain $\Omega\subset\R^m$ the celebrated  Cwikel-Lieb-Rozenbljum inequality gives a bound on the number $n(\al)$ of eigenvalues of a Schr\"odinger operator $L=-\Delta-q$ (with Dirichlet boundary conditions) which are less than or equal to $\al\in\R$
\begin{equation}\label{eq:CRL}
n(\al)\leq C\int_\Om (q+\al)_{+}^{\fr{m}{2}} \quad \text{where $(q+\al)_+=\twopartdef{q+\al }{q+\al \geq 0}{0}{\text{otherwise.}}$}	
\end{equation}
P.~Li and S.-T.~Yau's elegant proof of this fact \cite{LY83} elucidates the role of the $W^{1,2}\hookrightarrow L^\frac{2m}{m-2}$ Sobolev embedding, so that the constant in the above estimate \eqref{eq:CRL} depends explicitly and only upon the associated Sobolev constant.

S.-Y.~Cheng and J.~Tysk \cite{CT94} proved that if $\Sigma^m\subset N^{n}$ is a minimal immersion (which we know admits a Michael-Simon Sobolev inequality) in a closed Riemannian manifold, and $L=-\nabla^2-Q$ is a shifted rough Laplacian on a Riemannian vector bundle $\mathcal{V}$ over $\Sigma$ (here, $Q\in\textrm{End}(\mathcal{V})$ is some symmetric endomorphism), then for $\Om\In \Sigma$ with zero Dirichlet boundary conditions
$$n(0)=\textrm{index}(L) + \textrm{nullity}(L)\leq C(m,N)(\textrm{rank}(\mathcal{V}))\int_\Om (\max \{1, q\})^{\fr{m}{2}} \id V_\Sigma,$$
where $q(p)$ is the largest eigenvalue of $Q$ at $p\in \Sigma$.
In particular, and importantly, if $\mathcal{V}$ is the normal bundle of $\Sigma$, and
$ QX=\sum_{i,j}\inner{A(E_i,E_j),X}A(E_i,E_j) + \sum_i\Rm^N(X,E_i)E_i $
where $E_i$ is an orthonormal frame of $\Sigma$ and $A$ is its second fundamental form, then $L$ is the Jacobi operator of the area functional $\A^m$
$$\ed^2\A^m(\Sigma)[X,X] = \int_\Sigma \lan LX, X \ra \id V_\Sigma$$ so that the above provides interesting spectral estimates on minimal hypersurfaces in terms of $q(p)^{m/2}\leq K_m(\abs{A}^m + \abs{\Rm^N}^{m/2})$.
Assuming a uniform curvature bound on $N$ (for example, if $N$ is closed), one has that the index plus nullity (=$n(0)$) of a minimal surface is bounded by its area and total curvature $\int_\Sigma\abs{A}^m$. 
\textit{A fortiori}, this $L$ is also the Jacobi operator of area when restricted to volume preserving deformations along constant mean curvature hypersurfaces $\Sigma^m \In N^{m+1}$, and so similar estimates hold for CMC hypersurfaces, at least when $m\geq 3$. See below for precise definitions and clarifications of the notion of CMC index that we use in this paper.  

\subsection{The case $m=2$}
If the inequality \eqref{eq:CRL} were true for $m=2$ then the Gauss equations $|A|^2=2\kappa^N(\T\Sigma)-2 K^\Sigma -\tfrac12 |H|^2$ and the Gauss-Bonnet formula would allow one to prove that the index of a minimal or CMC surface is bounded affine-linearly from above by the genus and area of the domain, up to a uniform constant depending only on the ambient manifold $N$ and the size of the mean curvature $|H|$.

An estimate like \eqref{eq:CRL} is false when $m=2$ for general Schr\"odinger operators as standard counter-examples show. One cannot expect such an inequality to hold for general $q$ when the right-hand side is the $L^1$ norm of $q$. The correct replacement space is that $q$ be in the Zygmund class $L\log L$ see e.g. \cite{Solo94}. When $q\sim |A|^2 + Rm^N$ one cannot use Gauss-Bonnet to relate the $L\log L$-norm of $q$ to the genus of $\Sigma$. 
Nevertheless, for branched minimal immersions, Norio Ejiri and Mario Micallef \cite{EM08} used the relationship between Dirichlet energy and the area functional along immersions in order to prove that the analytically-expected result does indeed hold inside arbitrary closed ambient manifolds. 

Given a Riemann surface $\Sigma$, let $E(u) = \fr12\int_\Sigma |u_x|^2 + |u_y|^2 \id x\ed y$ and $\A (u) = \int_\Sigma |u_x\wedge u_y|\id x\ed y$ be the Dirichlet energy and area functionals respectively, defined here using isothermal coordinates. It is well known and easy to see that $\A(u)\leq E(u)$ with equality if and only if $u$ is conformal. Equally well known is that they agree up to first order too when $u$ is conformal i.e. $\A(u)=E(u)$ and $\ed \A(u)[v] = \ed E(u)[v]$ for any $v\in\Gamma(u^\ast \T N)$. In particular conformal harmonic maps are the same as a branched minimal immersions. One of the key results of \cite{EM08} is the comparison of second variations \emph{at a branched minimal immersion $u$},
\begin{equation}\label{eq:confrough}
	\dform^2{\A}(u)[v,v] = \dform^2{E}(u)[v,v] - 4\int_\Si\abs{\mu}^2\id\Si,
\end{equation}  
which holds for any variation $v\in \Gamma(u^*\T N)$, and where $\mu$ is a quantity depending on $v$ that measures the ``infinitesimal conformal deficit'' of the variation $v$. This $\mu$ is zero for instance along a variation which does not infinitesimally move the induced conformal structure of the domain, see \eqref{eq:infcon} and the surrounding discussion.  

As a corollary of the above comparison of the second variations of area and Dirichlet energy, Ejiri-Micallef prove the following comparison between the index $i_\A$ of the area functional and the index $i_E$ of the energy functional at a conformal harmonic map $u$. 

\begin{theorem}[Ejiri-Micallef: Theorem 1.1 in \cite{EM08}]\label{thm:ejiri-micallef}
Let $u:\Sigma_g\to N$ be a (possibly branched) minimal immersion of a closed Riemann surface of genus $g$ into a Riemannian manifold $N^n$. Then  
$$i_E\leq i_{\A} \leq i_E + r, \,\,\,\,\, \text{where} \,\,\,\,\, r=\threepartdef{6g-6-2b}{b\leq 2g -3}{4g-2+2\left[\frac{-b}{2}\right]}{2g-2\leq b \leq 4g - 4}{0}{b\geq 4g-3,}$$
and $b$ is the total number of branch points of $u$ counted with multiplicity, and $[x]=\max_{\Z}\{k: k\leq x\}$. 
\end{theorem}
In particular, even though there are no suitable analytical estimates on the index of $\ed^2 \A$ from above, the second variation of $E$ along a conformal harmonic map $u$ is of the form $-\D^2 - Q$ where $Q$ is uniformly bounded in $L^\infty$ by the ambient curvatures of $N$ and one may proceed to bound the $E$-index of $u$ by the area of its image, at least when one restricts to a conformal harmonic map. The above theorem thus provides an affine-linear upper bound on the area-index of a minimal immersion by the genus and area of the image.

The above is underpinned by the following fundamental heuristic from Ejiri-Micallef: Given a branched minimal immersion, we have $\A = E$ and in general $\A \leq E$. Any variation which decreases $E$ must therefore also decrease $\A$. On the other hand if a variation of $u$ continues to be a conformal map relative to the fixed Riemann surface $\Sigma$, the two functionals $\A$ and $E$ coincide along the variation, so the only possibility for a variation to decrease $\A$ but not necessarily $E$ is if the induced metrics along the variation are changing the conformal class. These variations infinitesimally form a vector space which Ejiri-Micallef show has dimension $r$, which is as expected the dimension of the moduli space of the underlying surface.  

\paragraph{Terminology}	For a constant $h$, we call $u:\Sigma \to N^3$ a \emph{branched CMC $h$-immersion} (or CMC $h$-immersion for short) if $u:\Sigma \to N^3$ is some branched conformal immersion from a closed Riemann surface into some oriented Riemannian three-manifold $N$ parametrising a CMC surface whose mean curvature vector is $H=h\nu$ where, in isothermal coordinates $(u_x,u_y,\nu)$ is positively oriented in $N$. We let $g$ be the genus of $\Sigma$ and $b$ the total number of branch points counted with multiplicity. 
	
\

 In particular a CMC $h$-immersion $u:\Sigma \to N$ is a critical point of the following functional 
$$\A_h(u) = \A(u) + V_h(u)$$
and $V_h$ is an enclosed volume with oriented weight $h$. There is, in general, no canonical form for an enclosed volume in $N^3$, but it is easy to see that there is a good choice if $u(\Sigma)$ misses a small ball in $N$: for any $\eps < inj_N$ and $B_\eps(x)\In N$ if we set
\begin{equation}\label{eq:alH}
\al_h = h\ast \ed f \in \Gamma(\wedge^2 \T^\ast (N\sm B_\eps(x)) \,\, \text{where} \,\,  f\in C_0^\infty(N\sm B_\eps(x)) \,\, \text{solves} \,\,  \Dl f = 1 	
\end{equation}
then $\ed \al_h = h\ed V_N$. 
Thus given any $u\in W^{1,2}(\Sigma,N\sm B_\eps(x))$ we may define
$V_h(u) : = \int_\Sigma u^\ast \al_h.$
In particular if $u$ is a CMC $h$-immersion then it will certainly avoid some such $B_\eps(x)\In N$ and we may consider it to be a critical point of $\A_h$ above (see Appendix \ref{app:full-sec-var-area}).
When $N=\R^3$ we have a globally-defined $\al_h = \tfrac{h}{3}(x_1 \ed x_2 \wedge \ed x_3 - x_2 \ed x_1 \wedge \ed x_3 + x_3 \ed x_1\wedge \ed x_2)$.  

We could equally well consider a branched CMC $h$-immersion to be a conformal critical point of 
$$E_h(u) = E(u) + V_h(u)$$ 
and in either case, given a CMC $h$-immersion we will implicitly assume that it is a critical point of both $\A_h$ and $E_h$ for an appropriate $V_h$. 

As to be expected (see \cite{bc3} for the Euclidean derivation and \eqref{eqn:second-var-area-h} in Appendix \ref{app:full-sec-var-area}) along a CMC $h$-immersion the second variation of $\A_h$ only sees the normal part of the variation and in particular if $s= f \nu = v^\bot$, $f=\lan v, \nu \ra$ we have  
\begin{equation}\label{eq:normvar}
\ed^2 \A_h(u)[v,v] = \ed^2 \A_h(u)[s,s] =  \ed^2 \A_h(u)[f,f]=\int_\Sigma f L f \id \Sigma	
\end{equation}
where $L=-\Dl - (|A|^2 + Ric_N(\nu,\nu))$. The `weak' CMC-index $i_h$, is equivalent to the number of ways we can reduce the area $\A(u)$ \emph{amongst volume preserving deformations}. This quantity is equivalent to  
$$i_h(u) :=\max_{V\In C^\infty(\Sigma)}\{\dim V : \ed^2\A_h(u)[f,f] < 0 \,\,\,\text{and} \,\,\int_\Sigma f\id \Sigma =0 \,\,\, \forall f\in V\},$$
where $V$ is a linear subspace. In this paper we will estimate 
$$i_{\A_h} : = \max_{V\In C^\infty(\Sigma)}\{\dim V : \ed^2\A_h(u)[f,f] < 0  \,\,\, \forall f\in V\} = \textrm{index}(L) \quad \text{and} \quad n_{\A_h} = \textrm{nullity}(L).$$
We recall that the index $i_{\A_h}$ of the operator $L$ differs from $i_h$ by at most $1$ (see e.g. \cite{bourni-sharp-tinaglia-2022}).  

One of the main auxiliary results of this paper, crucial to obtaining an index estimate for immersed CMC surfaces in arbitrary three-manifolds, is that the second-order comparison of area and Dirichlet energy remains true along conformal maps which do not necessarily parametrise minimal surfaces. In particular \eqref{eq:confrough} holds along \emph{all} conformal maps into arbitrary Riemannian manifolds $N^n$ (i.e. there is no constraint required on the mean curvature), see Section \ref{sec:comparison-sec-var-area-energy}. We therefore have exactly a comparison for the energies $\A_h$ and $E_h$. From this point we are able to utilise the approach of Ejiri-Micallef in order to show:     

\begin{theorem}\label{thm:main-theorem}
Suppose that $u:\Sigma_g^2\to  N^3\emb \R^d$ is a branched CMC $h$-immersion (see terminology above) from a closed Riemann surface of genus $g$ to an orientable $3$-manifold $N$, itself isometrically embedded in $\R^d$. Then there exists some uniform $C<\infty$ so that 
$$ i_{\A_h} + n_{\A_h} \leq C(4J^2+h^2)\Area(\Sigma) + r $$
where $J$ is the sup-norm of the largest eigenvalue of the second fundamental form $\sff$ of the embedding $N\into\R^d$ and $r$ is a purely topological quantity, exactly as in Theorem \ref{thm:ejiri-micallef} above: if $b$ is the total number of branch points counted with multiplicity then 
$$r=\threepartdef{6g-6-2b}{b\leq 2g -3}{4g-2+2\left[\frac{-b}{2}\right]}{2g-2\leq b \leq 4g - 4\qquad [x]=\max_{\Z}\{k: k\leq x\} }{0}{b\geq 4g-3.}$$
If in addition $u$ is totally umbilic then
$$ i_{\A_h} + n_{\A_h} \leq C(4J^2+h^2)\Area(\Sigma). $$
\end{theorem}
\begin{remark}
$C$ can be estimated explicitly by $\tfrac{60}{\pi}$, but it is most likely far from optimal as demonstrated by considering round spheres in $\R^3$ or $\Sp^3$. In the presence of non-negative ambient (sectional) curvatures the \emph{order} of the estimate (linear in terms of Willmore energy and genus) is certainly optimal --- see discussion and examples below.

In the case that $h=0$ this of course reduces to estimates obtained by Ejiri-Micallef e.g. \cite[Theorem 4.3]{EM08}, where here we show the specific dependence on the embedding of $N$ in $\R^d$.  	
\end{remark}

Recall the unduloid and nodoid CMC surfaces in $\R^3$ which lie in the same one-parameter family of rotationally symmetric and periodic immersions known collectively as Delaunay surfaces. To get from the (embedded) unduloids to the (non-embedded) nodoids, the family passes through a singular configuration of an infinite family of spheres of the same radius with centres along the $x$-axis (the axis of rotation), touching tangentially. We thus refer to each sphere-like portion of a Delaunay surface as a bulge or lobe and the connecting region as a neck in the below. Close to the singular limit we note that the Willmore energy of a (given region of a) Delaunay surface is roughly equivalent to the number of bulges in that region.

In a three-torus $T^3=\frac{[0,1]^3}{\sim}$, one can produce a singly-lobed Delaunay torus $u_1:T^2\to T^3$ by projecting (the fundamental piece of) an appropriately scaled Delaunay surface in $\R^3$ onto $T^3$.
For Delaunay tori with more lobes, one proceeds in the obvious way by
taking $k$ lengths of the fundamental piece of the Delaunay surface and scale by $1/k$, then immerse into the fundamental piece of $T^3$ so that under the projection $\pi:\R^3\to T^3$, one has a $k$-lobed Delaunay torus (with $k$ lobes and $k$ necks). The $x$-component of the unit normal to these surfaces is a Jacobi field, so the Courant nodal domain theorem tells us that the index of these surfaces asymptotically grows at least affine-linearly with twice the number of lobes, $\indx(u_k)\geq 2k-2$. We also have $\Area(u_k)=k^{-1}\Area(u_1)$ (the area of the fundamental piece, of which we have $k$ copies, scales by $k^{-2}$), and the mean curvature $h_k=kh_1$.

In this way, for large $k$, $\indx(u_k)$ and Willmore energy $h_k^2\Area(u_k)$ grow affine-linearly in $k$. In particular the fact that the Willmore energy appears linearly as an upper bound in our estimate is optimal for large index.

This optimality is also shown by considering families of Wente tori (resp. rotationally symmetric CMC tori) in $\R^3$ (resp. $\Sp^3$). Rossman \cite{rossman-2001-wente-tori} and Rossman-Sultana \cite{rossman-sultana-2007} have proven that the index of these CMC tori grows affine linearly with twice the number of ``bulges''. In both cases these surfaces look like a piece of a Delaunay surface, bent around and with ends glued together. The terms bulges and necks are similarly interpreted as above in these papers (indeed this is why we use it here). See the descriptions in \cite{rossman-2001-wente-tori,rossman-sultana-2007} and the references therein for more precise definitions. In both cases one may easily infer that the number of bulges is roughly equivalent to the Willmore energy of the surface, which in the case of $\Sp^3$ takes the form $(4+ h^2){\rm Area}(\Sigma)$. 

There are several interesting general lower-bounds on the index of a CMC-surface in Riemannian $3$-manifolds in terms of genus see e.g. \cite{NH19, cavalcante-oliveira-2020} most of which build on the works of Ros \cite{ros9}, Savo \cite{Savo10} and more recently \cite{acs_18}. For instance in a space form it is shown \cite{cavalcante-oliveira-2020} that the weak CMC index grows at least as fast as $\frac{g}{3+c}$ where $c$ is the curvature of the ambient space. Given the examples of arbitrary genus constructed by Kapouleas \cite{kap2} in Euclidean space, and Heller-Heller-Traizet \cite{heller-heller-traizet-2022} in $\Sp^3$ there are certainly CMC surfaces whose index diverges at least linearly with genus and thus the order of the estimate is optimal in terms of genus. 

Given the above discussion we pose the following conjecture, which should be compared with the conjecture of Schoen and Marques-Neves (see \cite{acs_18}) 
\begin{conj}
Let $u:\Sigma^2\to  N^3$ be a branched CMC $h$-immersion and suppose $N^3$ has strictly positive sectional curvature. Then there exists some $C=C(N)<\infty$ so that 
$$C \left((1+h^2)\Area(\Sigma) + g \right) \leq   i_{\A_h} + n_{\A_h}.$$
In the case that $N=\R^3$ or $N=T^3$ is a flat three torus, we expect 
$$C \left(h^2\Area(\Sigma) + g \right) \leq   i_{\A_h} + n_{\A_h}.$$
\end{conj}

In the presence of strictly negative sectional curvatures the estimate given in Theorem \ref{thm:main-theorem} is seen to be highly sub-optimal by considering spheres of large radius in a hyperbolic three-manifold. All CMC spheres in a hyperbolic three manifold are round (and the boundaries of geodesic balls) \cite{hf1,Chern83} and satisfy $i_{\A_h} + n_{\A_h} = 4$, but the Willmore energy grows exponentially with the radius: Given an arbitrary radius, there are closed hyperbolic three-manifolds admitting a CMC sphere of this radius and we always have that the mean curvature satisfies $h^2>4$ and $h^2\to 4$ as the radius diverges as can be checked directly.  

In the presence of sufficiently negative ambient sectional curvatures relative to the size of the mean curvature, we are able to obtain an index bound purely in terms of genus. This bound still does not apply to CMC spheres in hyperbolic space, however, showing that the hypothesis $h^2\leq 4|\kappa_0|$ in the below is indeed optimal: 
\begin{theorem}[cf. Corollary 4.1 in \cite{EM08}]\label{thm:curv}
Suppose that $u:\Sigma^2\to  N^3$ is a branched CMC $h$-immersion inside an orientable $3$-manifold $N$ whose sectional curvatures satisfy $\kappa^N\leq \kappa_0<0$.
If $h^2\leq 4|\kappa_0|$ then either
\begin{enumerate}
	\item $  i_{\A_h} + n_{\A_h}\leq r, $
where $r$ is a purely topological quantity defined in Theorem \ref{thm:main-theorem} or 
\item $u$ is totally umbilic, $h^2\equiv 4|\kappa_0|$, $Ric_N(\nu,\nu)\equiv -2|\kappa_0|$, $i_{\A_h}=0$ and $n_{\A_h} = 1$.
\end{enumerate}  
\end{theorem}

\paragraph{Acknowledgements}
The second named author is indebted to Mario Micallef for his encouragement and interest in this work.
The second author was supported by the EPSRC grant EP/W026597/1. 
The first author was supported by the EPSRC grant EP/W523860/1, project reference 2758306.

\section{Comparison of second variations of area and energy at a conformal map}\label{sec:comparison-sec-var-area-energy}

Consider a smooth immersion $u:\Sigma\to N^n$ from a closed manifold $\Sigma$ into a Riemannian manifold $N$. We equip $\Sigma$ with the pulled-back metric $u^\ast g$, which we still denote by $g$, making $u$ an isometric immersion. The $g$-inner product is denoted $\lan\cdot,\cdot\ra $ and the induced volume form is denoted $\ed \Sigma$.
If $v\in \Gamma (u^\ast \T N)$ is an infinitesimal variation, which we decompose via $v=\sigma+s$ with $s=v^\bot \in \Gamma( \mathcal{V}\Sigma)$ and $\sigma = v-s \in \Gamma(\T \Sigma)$, then (see \cref{app:full-sec-var-area}) we show that the Hessian of area in direction $v$ is as follows, noting that for us it is crucial both that $u$ is not necessarily minimal and $v$ is not necessarily normal:
\begin{eqnarray*}
\ed^2 \A^m (u)[v,v] &=& \int_{\Sigma} |\Db s|^2 -\sum_{i,j} \lan A(E_i,E_j), s \ra  ^2 - \sum_i Rm^N(s,E_i,E_i,s) + \lan s, H\ra^2 \\ 
&&    	\qquad + \lan A(\sigma,\sigma), H\ra + 2 \lan \D^\bot_\sigma s, H\ra \id \Sigma
\end{eqnarray*}
where $A$ is the second fundamental form, $H=\text{tr} A$ is the mean curvature vector and $\{E_i\}$ is an arbitrary orthonormal frame at each point.
Here we define the Hessian in the usual way: if $u_t:\Si\times(-\eps,\eps)\to N$ is a one-parameter family of variations of $u_0=u$ with $\pd_t u_t |_{t=0} = v$, then
	$$ \ed^2 \A^m(u)[v,v] = \pl{^2}{t^2}\vlinesub{t=0}\A^m(u(\cdot,t)) - \ed \A^m(u)\sbrak{\pl{^2}{t^2}u(\cdot,0)},$$
where (as in the Appendix) by $u_{tt}$ we mean $\nabla_{u_t} u_t$.
For a derivation of the above under the additional assumption that $u$ is minimal see \cite{spivak-dg4-1999}. For a derivation in full generality see \cref{app:full-sec-var-area}. 
We further remind the reader that if $E(u) = \fr12\int_\Sigma |\ed u|^2\id \Sigma $ for an arbitrary (not necessarily harmonic) map $u:(\Sigma,h)\to (N,g)$ between Riemannian manifolds, we have  
\begin{eqnarray*}
	\ed^2 E(u)[v,v]=  \int_\Sigma |\D v|^2 - \tr_h  Rm^N(v,\ed u,\ed u,v)\id \Sigma.
\end{eqnarray*}
In the above $\D$ and $\D^\bot$ denote the induced connections on $u^\ast \T N$  and $\mathcal{V}\Sigma$ respectively. 

We now restrict to dimension $m=2$. Let $\Sigma$ be a closed Riemann surface and $u:\Sigma \to N$ a (possibly branched) smooth conformal immersion.
We will use the notation $\lan,\ra$ as the real inner product $g$ on $N$, extended complex bi-linearly to $\T_\C N =\T N \otimes_{\R} \C $. We use the same notation for $u^\ast g$ the pulled-back metric on $u^\ast (\T_\C N)$. In a local complex coordinate $z=x+iy$, we use the notation $\de_z=\fr12(\de_x - i \de_y)$ and  $u_z=\fr12(u_x-iu_y)$. Defining $\lambda$ so that $e^{2\la}=2|\dz|_g^2=2|u_z|_g^2$ we have $u^\ast g = e^{2\la}(\ed x^2 + \ed y^2)$ and $\ed \Sigma = e^{2\la}\ed x \ed y$. $\D$ and $\D^\bot$ will be extended complex linearly as required.  

As in Ejiri-Micallef \cite{EM08} we let $\xi$ denote the ramified tangent bundle so that $u^\ast(\T N) = \xi \oplus \mathcal{V}\Sigma$ and we split $v\in \Gamma(u^\ast \T N)$ as above via $v=\sigma + s$ for $\sigma\in \Gamma(\xi)$ and $s\in \Gamma(\mathcal{V}\Sigma)$ noting that the complex structure on $\Sigma$ allows us to further decompose $\xi_\C = \xi \otimes_\R \C = \xi^{1,0} + \xi^{0,1}$.    

Given $\sigma \in \Gamma(\xi)$ we may decompose, away from the branch points, 
$\sigma = 2e^{-2\la}\lan \sigma , u_{\zb}\ra u_z + 2e^{-2\la}\lan\sigma ,u_z\ra u_{\zb} = \sigma^{1,0} + \sigma^{0,1}.$ Similarly we define the related $\sigma_\Sigma$ below which we consider to be a section of $\T_\C \Sigma$ away from the branch points: $\sigma_\Sigma := 2e^{-2\la}\lan \sigma , u_{\zb}\ra \de_z + 2e^{-2\la}\lan\sigma ,u_z\ra \dzb = \sigma_\Sigma^{1,0} + \sigma_\Sigma^{0,1}$.

In this setting, the second variation of area can be written (as may be checked directly): 
\begin{eqnarray}\label{eq:d2A}
\ed^2\A (u)[v,v]&=& \int_\Sigma 4e^{-2\la}|\D_\dz ^\bot s|^2 - 8e^{-4\la}|\lan s,A(u_z,u_z)\ra|^2 + \fr12\lan s,H\ra^2 -4e^{-2\la}Rm^N(s,u_z,u_{\zb}, s)  \nn \\
&&\qquad +  \lan A(\sigma,\sigma), H\ra + 2 \lan \D^\bot_{\sigma_\Sigma} s, H\ra \id \Sigma  \nn\\
&=& \int_\Sigma 4|\D_\dz ^\bot s|^2 - 8e^{-2\la}|\lan s,A(u_z,u_z)\ra|^2 + \fr12e^{2\la}\lan s,H \ra^2 -4Rm^N(s,u_z,u_{\zb}, s) \nn \\
&&\qquad +  e^{2\la}\lan A(\sigma,\sigma), H\ra + 2e^{2\la} \lan \D^\bot_{\sigma_\Sigma} s, H\ra \id x\ed y
\end{eqnarray}
and the second variation of $E$ (with $h=u^\ast g$ above) looks like 
\begin{eqnarray}\label{eq:d2E}
	\ed^2 E(u)[v,v]&=&  \int_\Sigma 4e^{-2\la}|\D_\dz v|^2 - 4 e^{-2\la}Rm^N(v,u_z,u_{\zb}, v)\id \Sigma \nn\\
	&=&  \int_\Sigma 4|\D_\dz v|^2 - 4 Rm^N(v,u_z,u_{\zb}, v)\id x\ed y. 
\end{eqnarray}

\paragraph{Infinitesimal conformal variations:} Following \cite{EM08} we will linearise around a conformal variation to find an infinitesimal formulation. If $u_t:\Sigma \to N$ is a one-parameter family of branched conformal immersions with $u_0=u$, $\de_t u_t \vline_{t=0} = v$ then we must have $\lan \dz u_t, \dz u_t\ra = 0$ for all small $t$, and thus
$$ \lan \D_\dz v, u_z\ra =0. $$

\

Say that $v\in \Gamma(u^\ast\T N)$ is an infinitesimal variation through conformal immersions if, with $v=\sigma+s$
$$ 0=(\D_\dz^\top \sigma+(\D_\dz s)^{\top} )^{0,1} $$
Since, using that $\lan u_z, u_z\ra = \lan u_{\zb},u_{\zb}\ra\equiv 0$,
\begin{equation}\label{eq:sig}
\rbrak{\D^\top_\dz \sigma^{1,0}}^{0,1} = 0\quad\text{and}\quad\rbrak{\D^\top_\dz \sigma^{0,1}}^{1,0}=0,
\end{equation}
we have that $v=\sigma+s$ is an infinitesimal variation through conformal immersions if 
$$ \D^\top_\dz \sigma^{0,1} + ((\D_\dz s)^{\top} )^{0,1}=0. $$

Inspired by this, and noting that (using $A(u_z,u_{\zb}) = \fr14 e^{2\la} H$) 
\begin{equation}\label{eq:dstan}
(\D_{\dz} s)^\top = -2e^{-2\la} \lan s, A(u_z,u_z)\ra u_{\zb} - \fr12\lan s, H\ra u_z	
\end{equation}
if we are given an arbitrary variation $v\in \Gamma(u^\ast \T N)$ which we decompose via $v=\sigma+s$ with $s=v^\bot \in \Gamma(u^\ast \mathcal{V}\Sigma)$ and $\sigma \in \Gamma(\xi)$, we refer to 
\begin{equation}\label{eq:infcon}
\eta = \D^\top_\dz \sigma^{0,1}-2e^{-2\la}\lan s,A(u_z,u_z)\ra u_{\zb}
\end{equation}
as the infinitesimal conformal defect. Let also $\mu=\eta\otimes \ed z\in \Gamma(\xi^{0,1} \otimes (\T^\ast\Sigma)^{1,0})$. Note also that, in view of \eqref{eq:dstan}, \eqref{eq:infcon} is the same as that appearing in \cite[equation (2.14)]{EM08} once we set $H=0$ as is done there. We remind the reader that the novelty of the below is that it holds for \emph{any} conformal map $u$ without constraint on its mean curvature $H$.
Related inequalities, assuming further assumptions on $u$ and $\Sigma$, can be found in \cite{cheng-zhou-2023,gao-zhu-2024}.
 
\begin{theorem}[cf. Theorem 2.1 in \cite{EM08}]\label{thm:second-var-area-energy}
Let $\Sigma$ be a closed Riemann surface and $N^n$ a closed Riemannian manifold. Suppose that $u:\Sigma\to N$ is a branched conformal immersion and $v\in \Gamma (u^\ast \T N)$. Then decomposing $v=\sigma + s$ for $\sigma\in \Gamma(\xi)$ and $s\in \Gamma(\mathcal{V}\Sigma)$ we have 
\begin{eqnarray*}
\ed^2 \A(u)[v,v] &=& \ed^2 E(u)[v,v]-8\int_\Sigma |\eta|^2\id x\ed y \\
&=&  \ed^2 E(u)[v,v]-4\int_\Sigma |\mu|^2\id \Sigma  	
\end{eqnarray*}
where $\eta$ is the infinitesimal conformal defect (see \eqref{eq:infcon}).  
\end{theorem}

A surface is totally umbilic if and only if $A(u_z,u_z)\ed z\tensor\ed z\equiv0$.
For such surfaces, given a normal variation $v=s$ (so $\sigma = 0$) one sees from \eqref{eq:infcon} that $\eta = 0$.
We have therefore the following corollary.

\begin{corollary}
\label{cor:totally-umbilic}
Let $\Sigma$ be a closed Riemann surface and $N^n$ a closed Riemannian manifold.
Suppose that $u:\Sigma\to N$ is a branched conformal immersion which is totally umbilic.
Then, for every $s\in\Gamma(\mathcal{V}\Sigma)$ we have that
$$ \ed^2 A(u)[s,s] = \ed^2 E(u)[s,s]. $$
\end{corollary}

\begin{proof}[Proof of Theorem \ref{thm:second-var-area-energy}]
When integrating against $\ed x \ed y$, the second variation of energy integrand may be written (see \eqref{eq:d2E})  
\begin{eqnarray}\label{eq:2ndvarenint}
4|\D_\dz v|^2 - 4Rm^N(s,u_z,u_{\zb}, s) &=& 4|\D_\dz^\bot s|^2 +  4|(\D_\dz s)^{\top}|^2+ 4|\D_\dz \sigma|^2 \nn   \\ 
&& + 4\left(\lan\D_\dz s,\D_\dzb \sigma\ra + \lan \D_\dzb s,\D_\dz\sigma \ra\right)	- 4Rm^N(v,u_z,u_{\zb}, v) \nn  \\
&= &4|\D_\dz^\bot s|^2 +  (i) + (ii) - 4Rm^N(s,u_z,u_{\zb}, s) 	
\end{eqnarray}	
where 
$$(i) = 4|(\D_\dz s)^{\top}|^2+ 4|\D_\dz \sigma|^2 - 4Rm^N(\sigma,u_z,u_{\zb},\sigma)$$
and 
$$(ii) = 4\left(\lan\D_\dz s,\D_\dzb \sigma\ra + \lan \D_\dzb s,\D_\dz\sigma \ra\right)- 4Rm^N(\sigma,u_z,u_{\zb},s)- 4Rm^N(s,u_z,u_{\zb},\sigma)$$
\paragraph{Term (i)}
One may check directly from \eqref{eq:dstan} that   
$$4|(\D_\dz s)^{\top}|^2=8e^{-2\la}|\lan s,A(u_z,u_z)\ra|^2 + \fr12 e^{2\la}\lan s,H\ra^2$$
meaning that 
\begin{eqnarray}\label{eq:i1}
	(i) &=& 8e^{-2\la}|\lan s,A(u_z,u_z)\ra|^2 + \fr12e^{2\la}\lan s,H\ra^2 + 4|A(u_z,\sigma)|^2 + 4|\D_\dz^\top \sigma|^2 -4Rm^N(\sigma,u_z,u_{\zb},\sigma)\nn  \\
	&=& 8e^{-2\la}|\lan s,A(u_z,u_z)\ra|^2 + \fr12e^{2\la}\lan s,H\ra^2 +  e^{2\la}\lan A(\sigma,\sigma),H\ra + 4Rm^N(\sigma,u_z,u_{\zb},\sigma) \nn \\
	&&  - 4Rm^\Sigma(\sigma,u_z,u_{\zb},\sigma) + 4|\D_\dz^\top  \sigma^{1,0}|^2+ 4|\D_\dz^\top  \sigma^{0,1}|^2-4Rm^N(\sigma,u_z,u_{\zb},\sigma) \nn \\
	&=&8e^{-2\la}|\lan s,A(u_z,u_z)\ra|^2 + \fr12e^{2\la}\lan s,H\ra^2 +  e^{2\la}\lan A(\sigma,\sigma),H\ra  \nn - 4Rm^\Sigma(\sigma,u_z,u_{\zb},\sigma)\\ 
	&& + 4|\D_\dz^\top  \sigma^{1,0}|^2+ 4|\D_\dz^\top  \sigma^{0,1}|^2 \end{eqnarray}
	where the middle line follows from the Gauss equations, using $|A(u_z,\sigma)|^2 = \lan A(u_z,\sigma),A(\dzb,\sigma)\ra$ and also \eqref{eq:sig}.  
Now, by the symmetries of the curvature tensor and the first Bianchi identity we see that 
\begin{eqnarray}\label{eq:i2}
-4Rm^\Sigma(\sigma,u_z,u_{\zb},\sigma) &=&  -4Rm^\Sigma(\sigma^{0,1},u_z,u_{\zb},\sigma^{1,0})\nn \\
&=&  4Rm^\Sigma(u_z,u_{\zb},\sigma^{0,1},\sigma^{1,0})	\nn \\
&=& 4\lan \D^\top_\dz\D^\top_\dzb \sigma^{0,1},\sigma^{1,0}\ra - 4\lan \D^\top_\dzb \D^\top_\dz\sigma^{0,1},\sigma^{1,0}\ra \nn \\
&=& 4\dz\lan \D^\top_\dzb \sigma^{0,1},\sigma^{1,0}\ra - 4 \dzb\lan \D^\top_\dz\sigma^{0,1},\sigma^{1,0}\ra \nn \\
&& -4|\D^\top_\dz \sigma^{1,0}|^2 + 4|\D^\top_\dz \sigma^{0,1}|^2.
\end{eqnarray}
Putting \eqref{eq:i1} and \eqref{eq:i2} together, the integral $\int_\Sigma (i) \id x\ed y$ becomes  
\begin{eqnarray}\label{eq:i3}
	\int_\Sigma   8e^{-2\la}|\lan s,A(u_z,u_z)\ra|^2 + \fr12e^{2\la}\lan s,H\ra^2 +  e^{2\la}\lan A(\sigma,\sigma),H\ra  + 8|\D_\dz^\top  \sigma^{0,1}|^2  \id x\ed y,
	\end{eqnarray}
where the $\dz$ term (and, similarly, the $\dzb$ term) disappears for the following reason:  we firstly have that the (temporarily-defined) one form $\om:=\inner{\nabla^T_{\dzb}\sigma^{0,1},\sigma^{1,0}}\ed\overline{z}$ is independent of co-ordinates, and so is in fact a globally defined one-form on $\Si$;
we secondly have, setting (temporarily) $f:=\dz\inner{\nabla^T_{\dzb}\sigma^{0,1},\sigma^{1,0}}$, that $\ed\om = f\ed z\wedge\ed\overline{z} = -2if\ed x\wedge\ed y$; finally then by Stokes' theorem we see the vanishing of the claimed terms.

From the definition of $\eta$ we have that 	
\begin{eqnarray*}
	8|\D_\dz^\top  \sigma^{0,1}|^2 &=& 8|\eta|^2 -16e^{-2\la}|\lan s,A(u_z,u_z)\ra|^2 \\
	&& + 16 e^{-2\la}\left(\lan \D^\top_\dz \sigma^{0,1},u_z\ra\lan s,A(u_{\zb},u_{\zb})\ra  + \lan\D^\top_\dzb \sigma^{1,0},u_{\zb}\ra\lan s,A(u_z,u_z)\ra\right) \\
	  &=& 8|\eta|^2 -16e^{-2\la}|\lan s,A(u_z,u_z)\ra|^2 \\
	  && + 8\lan s , A(\D^\top_\dz \sigma^{0,1}, u_{\zb})\ra  + 8 \lan s , A(\D^\top_\dzb \sigma^{1,0}, u_z)\ra \qquad \text{by \eqref{eq:sig}}  \\
	  &=& 8|\eta|^2 -16e^{-2\la}|\lan s,A(u_z,u_z)\ra|^2 \\
	  && - 8\lan (\D_\dzb s)^{\top} , \D^\top_\dz \sigma^{0,1}\ra  - 8 \lan (\D_\dz s)^\top , \D^\top_\dzb \sigma^{1,0}\ra. 
\end{eqnarray*}	
Thus \eqref{eq:i3} gives 
\begin{eqnarray}\label{eq:i4}
	\int_\Sigma (i)\id x \ed y &=& \int_\Sigma   -8e^{-2\la}|\lan s,A(u_z,u_z)\ra|^2 + \fr12e^{2\la}\lan s,H\ra^2 +  e^{2\la}\lan A(\sigma,\sigma),H\ra  + 8|\eta|^2 \nn \\
	&& \qquad  - 8\lan (\D_\dzb s)^{\top} , \D^\top_\dz \sigma^{0,1}\ra  - 8 \lan (\D_\dz s)^\top , \D^\top_\dzb \sigma^{1,0}\ra \id x\ed y. 
\end{eqnarray}

\paragraph{Term (ii)}
By straightforward decomposition of $\sigma$ and the symmetries of the curvature tensor, along with the first Bianchi identity, we have 
\begin{eqnarray}
	(ii) &=& 4\lan \D_\dz s, \D_\dzb \sigma^{1,0}\ra +  4\lan \D_\dz s, \D_\dzb \sigma^{0,1}\ra +  4\lan \D_\dzb s, \D_\dz \sigma^{1,0}\ra+  4\lan \D_\dzb s, \D_\dz \sigma^{0,1}\ra \nn \\
	&& - 4Rm^N(\sigma^{0,1},u_z,u_{\zb},s) - 4Rm^N(\sigma^{1,0}, u_{\zb}, u_z,s) \nn \\
	&=& 4\lan \D_\dz s, \D_\dzb \sigma^{1,0}\ra +  4\lan \D_\dz s, \D_\dzb \sigma^{0,1}\ra +  4\lan \D_\dzb s, \D_\dz \sigma^{1,0}\ra+  4\lan \D_\dzb s, \D_\dz \sigma^{0,1}\ra \nn \\
	&& + 4Rm^N(u_z,u_{\zb},\sigma^{0,1},s) + 4Rm^N( u_{\zb}, u_z,\sigma^{1,0},s) \nn \\
	&=& 4\lan \D_\dz s, \D_\dzb \sigma^{1,0}\ra +  4\lan \D_\dz s, \D_\dzb \sigma^{0,1}\ra +  4\lan \D_\dzb s, \D_\dz \sigma^{1,0}\ra+  4\lan \D_\dzb s, \D_\dz \sigma^{0,1}\ra \nn \\
	&& + 4\dz\lan s,\D_\dzb \sigma^{0,1}\ra - 4\lan \D_\dz s,\D_\dzb \sigma^{0,1} \ra - 4\dzb\lan s,\D_\dz \sigma^{0,1}\ra + 4 \lan \D_\dzb s,\D_\dz \sigma^{0,1} \ra  \nn \\ 
	&&+ 4\dzb\lan s,\D_\dz \sigma^{1,0}\ra - 4\lan \D_\dzb s,\D_\dz \sigma^{1,0} \ra - 4\dz\lan s,\D_\dzb \sigma^{1,0}\ra + 4 \lan \D_\dz s,\D_\dzb \sigma^{1,0} \ra \nn
\end{eqnarray}
meaning that 
\begin{eqnarray}\label{eq:ii}
	\int_{\Sigma} (ii) \id x \ed y &=& \int_\Sigma   8\lan \D_\dzb s, \D_\dz \sigma^{0,1}\ra +8\lan \D_\dz s, \D_\dzb \sigma^{1,0}\ra \id x \ed y \nn\\
	&=& \int_{\Sigma} 8\lan \D_\dzb^\bot  s , A(u_z, \sigma^{0,1})\ra + 8\lan \D_\dz^\bot s, A(u_{\zb},\sigma^{1,0})\ra \nn \\
	&&\qquad  +  8\lan (\D_\dzb s)^{\top} , \D^\top_\dz \sigma^{0,1}\ra  + 8 \lan (\D_\dz s)^\top , \D^\top_\dzb \sigma^{1,0}\ra \id x \ed y \nn \\
	&=& \int_{\Sigma} 8\lan \D^\bot_{\sigma^{0,1}_\Sigma}s, A(u_z,u_{\zb})\ra  + 8\lan \D^\bot_{\sigma^{1,0}_\Sigma}s, A(u_{\zb},u_z)\ra\nn \\
	&&\qquad  +  8\lan (\D_\dzb s)^{\top} , \D^\top_\dz \sigma^{0,1}\ra  + 8 \lan (\D_\dz s)^\top , \D^\top_\dzb \sigma^{1,0}\ra \id x \ed y \nn \\
	&=& \int_\Sigma 2e^{2\la}\lan \D_{\sigma_\Sigma}^\bot s, H\ra +  8\lan (\D_\dzb s)^{\top} , \D^\top_\dz \sigma^{0,1}\ra  + 8 \lan (\D_\dz s)^\top , \D^\top_\dzb \sigma^{1,0}\ra \id x \ed y \qquad\quad  
\end{eqnarray}
where the penultimate line follows from the definitions $\sigma_\Sigma^{1,0}=2e^{-2\la}\lan\sigma,u_{\zb}\ra \de_z$, $\sigma_\Sigma^{0,1} =2e^{-2\la} \lan\sigma,u_z\ra \dzb$, $\sigma^{1,0}=2e^{-2\la}\lan\sigma,u_{\zb}\ra u_z$ and $\sigma^{0,1} =2e^{-2\la} \lan\sigma,u_z\ra u_{\zb}$. The final line follows from $A(u_z,u_{\zb}) = \fr14 e^{2\la}H$. 

Putting \eqref{eq:d2A}, 
\eqref{eq:d2E} and \eqref{eq:2ndvarenint} together with \eqref{eq:i4} and \eqref{eq:ii} gives 
\begin{eqnarray*}
\ed^2 E(u)[v,v] &=& \int_{\Sigma} 4|\D_\dz^\bot s|^2 -8e^{-2\la}|\lan s,A(u_z,u_z)\ra|^2 + \fr12e^{2\la}\lan s,H\ra^2 +  e^{2\la}\lan A(\sigma,\sigma),H\ra  + 8|\eta|^2  \\
&&+2e^{2\la}\lan \D_{\sigma_\Sigma}^\bot s, H\ra	- 4Rm^N(s,u_z,u_{\zb}, s) \id x \ed y  \\
&=& \ed^2\A(u)[v,v] + 8\int_{\Sigma} |\eta|^2 \id x\ed y. 
\end{eqnarray*}\end{proof}

\section{The CMC index bounds: proofs of the main Theorems}\label{sec:cmc-index-bounds}

\begin{theorem}\label{thm:bound-index-cmc-energy}
Let $u$ be a CMC $h$-immersion.
Then $i_{\A_h} + n_{\A_h} \leq i_{E_h} + n_{E_h} + r$, where $r$ is as in Theorem \ref{thm:main-theorem}.
If in addition $u$ is totally umbilic then $i_{\A_h} + n_{\A_h} \leq i_{E_h} + n_{E_h}$.
\end{theorem}
\begin{proof}
This is analogous to the proof of Theorem 1.1 in \cite{EM08}, but we give an outline for the convenience of the reader. 

Let $V\In \Gamma(u^*\T N)$ be the non-positive eigenspace of $L$, meaning that $\text{dim}V = i_{\A_h} + n_{\A_h}$.
Let $\phi:V\to \xi^{0,1}\tensor(\T^*\Si)^{1,0}$ be the linear map $\phi(f):=2e^{-2\la}\inner{f\nu,A(u_z,u_z)}u_{\zb}\tensor\ed z$. The equation $\mu = 0$ (see Theorem \ref{thm:second-var-area-energy}) can be written 
\begin{equation}\label{eqn:the-de}
D\sigma^{0,1} = \phi(s),
\end{equation}
where $D:=\ed z\otimes\nabla^\top_{\pd_z}$ is defined globally.

By the Fredholm alternative, we may solve \eqref{eqn:the-de} if and only if $\phi(s)\perp\text{ker}D^*$, where
$$
D^\ast : \Gamma(\xi^{0,1} \otimes (\T^\ast\Sigma)^{1,0}) \to \Gamma(\xi^{0,1})
\quad\text{with}\quad
D^\ast = i \ast \db,
$$
and $\db$ is the induced connection $\db$-operator on $\xi^{0,1} \otimes (\T^\ast\Sigma)^{1,0}$.
Accordingly, linearly decompose $V=V_1\oplus V_2$ where
$$
V_1 = \{s : \phi(s)\in \left(\text{ker}D^*\right)^\perp \}
\quad\text{and}\quad
V_2 = \{s : \phi(s)\in \text{ker}D^*\},
$$
where we arrange that $\text{ker}\phi \In V_1$ so that $\phi|_{V_2}$ is injective. 
 
For each $s\in V_1$ we may find a corresponding $\sigma_s$ solving \eqref{eqn:the-de}.
Moreover, since \eqref{eqn:the-de} is linear we may arrange for $\sigma_s$ to depend linearly on $s$.
Then, for each $s\in V_1$ we have, by \eqref{eq:normvar} and Theorem \ref{thm:second-var-area-energy},
$$ \ed^2\A_h(u)[s,s] = \ed^2\A_h(u)[s+\sigma_s,s+\sigma_s] = \ed^2 E_h(u)[s+\sigma_s,s+\sigma_s], $$
and so straightforwardly we have $\text{dim}V_1\leq i_{E_h}+n_{E_h}$.
On the other hand, $\text{dim}V_2\leq\text{dim\,ker}D^*$ may be estimated using Riemann--Roch; see \cite[Proof of Theorem 1.1, p. 231]{EM08} for full details.
One obtains $\text{dim\,ker}D^*\leq r$ as required.

Finally, if $u$ is $h$-CMC and also totally umbilic ($A(u_z,u_z)\ed z\tensor\ed z\equiv0$), then $\phi$ is the zero map so $V=V_1$ and the result follows.
\end{proof}

\begin{remark}
By starting with $V$ being the negative eigenspace of $L$, an easy modification of the above proof gives a comparison of the indices $i_{\A_h}\leq i_{E_h} + r$, without the nullities.
We note, however, that it is not obviously the case that $i_{E_h}\leq i_{\A_h}$ at a $h$-CMC surface when $H=h\nu\neq 0$ since there may be purely tangential variations in the negative eigenspace of $\ed^2 E_h(u)$; by Theorem \ref{thm:second-var-area-energy}, and Corollary \ref{cor:a3},
$$ \ed^2 E_h(u)[\sigma,\sigma] = \int_\Si\! 4\abs{D\sigma^{0.1}}^2 - \langle A(\sigma,\sigma),H\rangle\,\ed\Si. $$
\end{remark}

The proofs of Theorems \ref{thm:main-theorem} or \ref{thm:curv} follow from Theorem \ref{thm:bound-index-cmc-energy} combined with Theorems \ref{thm:energy-bound} or \ref{thm:index-bound-curvature-assumptions} respectively, which we state below and prove in the next section. 

\begin{theorem}\label{thm:energy-bound}
Suppose that $u:\Sigma\to  N^3\emb \R^d$ is a branched CMC $h$-immersion from a closed Riemann surface to an orientable $3$-manifold $N$, itself isometrically embedded in $\R^d$. Then there exists some uniform $C<\infty$ so that 
	$$ i_{E_h} + n_{E_h} \leq C\rbrak{4J^2 + h^2}\Area(\Si), $$
where $J$ is the sup-norm of the largest eigenvalue of the second fundamental form $\sff$ of the embedding $N\into\R^d$.
\end{theorem}

\begin{theorem}\label{thm:index-bound-curvature-assumptions}
Suppose that $u:\Sigma^2\to  N^3$ is a branched CMC $h$-immersion inside an orientable $3$-manifold $N$ whose sectional curvatures satisfy $\kappa^N\leq \kappa_0<0$.
If $h^2\leq 4|\kappa_0|$ then either
\begin{enumerate}
	\item $  i_{E_h} + n_{E_h}=0, $ or 
\item $u$ is totally umbilic, $h^2\equiv 4|\kappa_0|$, $Ric_N(\nu,\nu)\equiv -2|\kappa_0|$, $i_{\A_h}=0$ and $n_{\A_h} = 1$.
\end{enumerate}  
\end{theorem}

\subsection{Proofs of Theorems \ref{thm:energy-bound} and \ref{thm:index-bound-curvature-assumptions}}

\begin{proof}[Proof of Theorem \ref{thm:energy-bound}]
As derived in \cite[\textrm{II}.1.2]{dalio-gianocca-riviere-2023} we may write 
$$\ed^2 V_h (u)[v,v]:=\int_\Sigma \ed \al_h(v,\D_{\de_x}v,u_y) + \ed \al_h (v,u_x,\D_{\de_y}v) + \rbrak{\D_v \ed \al_h }(v,u_x,u_y)\id x\ed y$$
and since $\ed \al_h = h \ed V_N$ the last term vanishes. 

For simplicity let $\om = \ed V_N$. Using the Peter-Paul inequality, we have, for any $\eps>0$,
\begin{align}
|e^{-2\la} h(\om(v,\D_{\de_x}v,u_y) + &\om(v,u_x,\D_{\de_y}v))| \leq
	h\rbrak{e^{-\la}\abs{\D_{\de_x}v}\abs{v}e^{-\la}\abs{u_y}
		+ e^{-\la}\abs{\D_{\de_y}v}\abs{v}e^{-\la}\abs{u_x}}  \nn \\
		&\leq \frac{\eps}{2}h^2\abs{v}^2 + \frac{1}{2\eps}e^{-2\la}\abs{\D_{\de_x}v}^2
			+ \frac{\eps}{2}h^2\abs{v}^2 + \frac{1}{2\eps}e^{-2\la}\abs{\D_{\de_y}v}^2  \nn \\
		&= \eps h^2\abs{v}^2 + \frac{1}{2\eps}\abs{\nabla v}^2. \label{eq:pp}
\end{align}
Therefore the second variation of energy may be written 
\begin{align*}
	\secondvar{E_h}{u}{v,v} &= \int_\Si \abs{\nabla v}^2 - \Rm(v,e^{-\la}u_x,e^{-\la}u_x,v) - \Rm(v,e^{-\la}u_y,e^{-\la}u_y,v) \\
	&\qquad\qquad + h\rbrak{\om(v,\nabla_{e^{-\la}\pd_x}v,e^{-\la}u_y) + \om(v,e^{-\la}u_x,\nabla_{e^{-\la}\pd_y}v)}
	\id\Si  \\
	& \geq \fr12 \int_\Sigma |\D v|^2 - 2\Rm(v,e^{-\la}u_x,e^{-\la}u_x,v) - 2\Rm(v,e^{-\la}u_y,e^{-\la}u_y,v) - 2h^2 |v|^2 \id \Sigma
\end{align*}
by \eqref{eq:pp} with $\eps =1$. 

We now note that by the Gauss equations:  
	$$\sup_{p\in\Si} \sup_{\substack{X,Y\in T_{u(p)}N,\\ \abs{X}=\abs{Y}=1}}\Rm(X,Y,Y,X) \leq J^2 $$
	which leaves
\begin{align*}
	\secondvar{E_h}{u}{v,v} &\geq \fr12\int_\Sigma |\D v|^2 - (4J^2 + 2h^2)|v|^2 \id \Sigma  \\
	&= \frac{1}{2}\int_\Si \inner{(-\overline{\Delta}- 4J^2 - 2h^2)v,v} \id\Si
\end{align*}
where $\overline{\Delta}$ is the rough Laplacian in $u^\ast \T N$.

This estimate immediately leads to
\begin{equation}\label{ineq:index-nullity-leq-count}
i_{E_h} + n_{E_h} \leq \indx(-\overline{\Delta}-4J^2 - 2h^2) + \nullity(-\overline{\Delta}-4J^2 - 2h^2) \leq \#\cbrak{\overline{\la}\leq 4J^2 +2 h^2},
\end{equation}
where $\ov\la$ are the eigenvalues of $-\overline{\Delta}$ (counted with multiplicity).

Now, for all $t>0$, we have
\begin{equation}\label{ineq:evals-leq-bar-heat}
\#\cbrak{\overline{\la}\leq 4J^2 + 2h^2}\leq \sum_i e^{-t(\overline{\la}_i-(4J^2 + 2h^2))} = e^{(4J^2 + 2h^2)t}\overline{k}(t),
\end{equation}
where $\overline{k}(t) = \sum_i e^{-\overline{\la}_i t}$ is the trace of the heat kernel of $-\overline{\Delta}$.

Similarly, the trace of the heat kernel related to the scalar Laplacian $-\Delta$ on $\Sigma$, is $ k(t) = \sum_i e^{-\la_i t}$ where $\{\la_i\}$ are the eigenvalues of $-\Delta$.

Now, by Theorem 2.1 of \cite{urakawa-1987}, we have that $\overline{k}(t)\leq3k(t)$ for all $t>0$, and so keeping in mind the estimates (\ref{ineq:index-nullity-leq-count}) and (\ref{ineq:evals-leq-bar-heat}), we have
\begin{equation}\label{ineq:index-nullity-leq-inf}
i_{E_h} + n_{E_h} \leq 3 \inf_{t>0} e^{(4J^2 + 2h^2)t}k(t).
\end{equation}

At this point we are required to estimate the heat kernel of $-\Delta$, and we do this \via\ an argument of Cheng--Tysk \cite[pp.~991 -- 993]{CT94}. We provide details since we  wish to keep track of constants more precisely.  

View $u$ as a map $\Si\to\R^d$, whose mean curvature is $\vec{H}:=H+\calj$, where $\calj=\tr_\Si\sff$ ($\sff$ is the second fundamental form of $N$ in $\R^d$) and $H$ is of course the mean curvature of $\Si$ in $N$: note that $|\vec{H}| = \sqrt{|H|^2 + |\calj|^2} \leq  \sqrt{h^2 + 4J^2}$. 

Applying the Michael-Simon-Sobolev inequality to a general $f\in W^{1,1}(\Si)$, we have
$$ \rbrak{\int_\Si f^2}^{1/2} \leq \frac{1}{\sqrt{2\pi}}\int_\Si\rbrak{\abs{\nabla f} + \sqrt{h^2 + 4J^2}|f|}. $$
Replace $f$ with $f^2$; an application of H\"older's inequality and the Peter-Paul inequality yields 
$$ \rbrak{2\pi\int_\Si f^4}^{1/2} \leq \int_\Si\rbrak{|\nabla f^2| + \sqrt{h^2 + 4J^2}f^2} \leq (1+\dl)\sqrt{h^2 + 4J^2}\int_\Si f^2 + \frac{1}{\dl\sqrt{h^2 + 4J^2}}\int_\Si\abs{\nabla f}^2. $$
	
An interpolation inequality ($\|f\|_{L^2}^3\leq \|f\|_{L^4}^{2}\|f\|_{L^1}$)  yields, for any $\dl >0$: 
\begin{equation}\label{ineq:cheng-tysk-sobolev}
\sqrt{2\pi}\frac{\rbrak{\int_\Si f^2}^{3/2}}{\int_\Si\abs{f}} \leq (1+\dl)\sqrt{h^2 + 4J^2}\int_\Si f^2 + \frac{1}{\dl\sqrt{h^2 + 4J^2}}\int_\Si\abs{\nabla f}^2.
\end{equation}
Now, consider $K(x,y,t)$ the heat kernel of the Laplace-Beltrami operator on $\Si$.
From, for example, \cite[p.~103]{schoen-yau-1994}, we know that
$$ K(x,y,t) = \sum_i e^{-\la_i t}\phi_i(x)\phi_i(y),
\qquad \lim_{t\downarrow0} K(x,y,t) = \delta_x(y)
$$
and that $K>0$ (see, for instance, \cite[Lemma~1, p.~99]{schoen-yau-1994}), where $\la_i$ are the eigenvalues of the Laplace-Beltrami operator on $\Si$ counted with multiplicity, and $\phi_i$ are the corresponding $L^2$-orthonormal eigenfunctions.
Moreover, since $\pd_t\int_\Si\!K(x,y,t)\id{y}=\int_\Si\!\Delta K(x,y,t)\id{y}=0$, we have that $\int_\Si\abs{K(x,y,t)}\id{y}=\lim_{t\downarrow0}\int_\Si K(x,y,t)\id{y}=1$.
Setting therefore $f(y):=K(x,y,t)$ in the estimate (\ref{ineq:cheng-tysk-sobolev}) we have
$$ \sqrt{2\pi}\rbrak{\int_\Si K^2\dform{y}}^{3/2} \leq (1+\dl)\sqrt{h^2 + 4J^2}\int_\Si K^2\dform{y} +\frac{1}{\dl\sqrt{h^2 + 4J^2}} \int_\Si \abs{\nabla_y K}^2 \id{y}. $$
But $K(x,x,t)=\int_\Si K(x,y,t/2)^2\id{y}$, and so
\begin{align*}
	\partial_tK(x,x,t) &= \int_\Si K(x,y,t/2)\rbrak{\partial_t K}(x,y,t/2)\id{y} \\
		&= \int_\Si K(x,y,t/2)\rbrak{\Delta_y K}(x,y,t/2) \id{y} \\
		&= -\int_\Si\abs{\nabla_y K(x,y,t/2)}^2 \id{y}.
\end{align*}
Therefore,
$$ \sqrt{2\pi}\dl\sqrt{h^2 + 4J^2}K(x,x,t)^{3/2} \leq \dl(1+\dl)(h^2 + 4J^2)K(x,x,t) - \partial_t K(x,x,t). $$

Setting $\phi(t) = K^{-1/2}(x,x,t)$ yields the ordinary differential inequality
$$  \sqrt{\pi}\dl\sqrt{\frac{1}{2}\rbrak{h^2 + 4J^2}} \leq \dl(1+\dl)\frac{1}{2}(h^2 + 4J^2)\phi + \phi', $$
the solution of which can be obtained by integration (noting that $\phi(0)=0)$, yielding, with $\al = \fr12(h^2 + 4J^2)$
$$ \frac{\sqrt{2\pi}}{(1+\dl)\sqrt{h^2 + 4J^2}}\frac{e^{\dl(1+\dl)\al t} - 1}{e^{\dl(1+\dl)\al t}}\leq \phi(t). $$
Therefore, replacing $\phi$ with its expression in terms of the heat kernel, and noting that $k(t)=\int_\Si K(x,x,t)\id{x}$, gives
$$ k(t) \leq \frac{(1+\dl)^2}{2\pi}\Area(\Si)(h^2 + 4J^2)\cdot\rbrak{\frac{e^{\dl(1+\dl)\al t} }{e^{\dl(1+\dl)\al t}-1}}^2. $$
Using this in the inequality (\ref{ineq:index-nullity-leq-inf}), we obtain
\begin{align}
	i_{E_h} + n_{E_h} &\leq \frac{3(1+\dl)^2}{2\pi}\Area(\Si)(h^2 + 4J^2) \inf_{t>0}\rbrak{\frac{e^{(\dl(1+\dl) + 2)\al t} }{e^{\dl(1+\dl)\al t}-1}}^2 \nonumber\\
	&= \frac{3}{2\pi}\Area(\Si)(h^2 + 4J^2) \fr{(\dl(1+\dl)+2)^{2+ \fr{4}{\dl(1+\dl)}}}{\dl^2 2^{\fr{4}{\dl(1+\dl)}}} \nn,
\end{align}
for any $\dl >0$. 

The expression in $\dl$ is less than $40$ when $\dl \approx 2.3$ and so we have an upper bound of the form: 
$$i_{E_h} + n_{E_h} \leq \fr{60}{\pi}\Area(\Si)(h^2 + 4J^2).$$ 

\end{proof}

\begin{proof}[Proof of Theorem  \ref{thm:index-bound-curvature-assumptions}]
 Again letting $\ed V_N = \om$
\begin{multline}
\secondvar{E_h}{u}{v,v} = \int_\Si \abs{\nabla v}^2 - e^{-2\la}\rbrak{\Rm(v,u_x,u_x,v) + \Rm(v,u_y,u_y,v)} \\ + e^{-2\la} h\rbrak{\om(v,\D_{\de_x}v,u_y) + \om(v,u_x,\D_{\de_y}v)} \id\Si \nn 
\end{multline}
Now, the fact that any sectional curvature $\kappa^N\leq \kappa_0<0$ gives that
$$ -e^{-2\la}\rbrak{\Rm(v,u_x,u_x,v) + \Rm(v,u_y,u_y,v)} \geq 2|\kappa_0|\abs{v}^2. $$
Setting $\eps=1/2$ in \eqref{eq:pp}  yields the estimate
$$ \secondvar{E_h}{u}{v,v} \geq  \int_\Si \rbrak{2|\kappa_0| - \frac{1}{2}h^2}\abs{v}^2\id\Si, $$
whence if $h^2\leq 4|\kappa_0|$, $\dform^2{E_h}(u)\geq0$. If $\ed^2 E_h>0$ we clearly have conclusion 1. 

On the other hand if there exists $v$ so that $\ed^2 E_h(u)[v,v]=0$, following the equality case back through the proof (and also in \eqref{eq:pp})
 we see that $h^2=4|\kappa_0|$ and up to constant $v=\nu$, $\D_{\de_x} \nu = -|\kappa_0|^\fr12 u_x$, $\D_{\de_y} \nu = -|\kappa_0|^\fr12 u_y$  and of course $Ric(\nu,\nu) \equiv -2|\kappa_0|$. This easily yields that $u$ is totally umbilic ($A(u_z,u_z) \equiv 0$) and that $|A|^2 \equiv 2|\kappa_0|$. In particular $\ed^2 \A_h(u)[1,1] = 0$ and so we must have $n_{\A_h}=1$ and $i_{\A_h} = 0$.   
\end{proof}

%--------------APPENDIX-----------------------------------------

\appendix 

\section{Area, conformally invariant variational problems, enclosed volume and their variations}\label{app:full-sec-var-area}
We first derive the full second variation of area along a (potentially) non-minimal immersion and along arbitrary (non-normal) variations. Of course the result holds under far fewer assumptions on $u$ and its variation $v$ but for simplicity we state the results for smooth immersions and variations. 
\begin{theorem}\label{thm:app1} 
	Let $u:\Sigma^m\to N^n$ be a smooth immersion from a closed manifold into some Riemannian manifold $(N,g)$. Suppose that $u(x,t):\Sigma\times (-\dl,\dl)\to N$ is a smooth one-parameter family of immersions and we denote $v=u_t(x,0)\in \Gamma(u^\ast\T N)$ which we decompose via $v=\sigma + s$ where $s=v^\bot\in \Gamma(u^\ast \mathcal{V}\Sigma)$ and $\sigma \in \Gamma(\T\Sigma)$, then 
	\begin{eqnarray*}
\ed^2 \A^m (\Sigma)[v,v] &=& \int_{\Sigma} |\Db s|^2 -\sum_{i,j} \lan A(E_i,E_j), s \ra  ^2 - \sum_i Rm^N(s,E_i,E_i,s) + \lan s, H\ra^2 \\ 
&&    	\qquad + \lan A(\sigma,\sigma), H\ra + 2 \lan \Db_\sigma s, H\ra \id \Sigma,
\end{eqnarray*}
where $\{E_i\}$ is an arbitrary orthonormal frame for $\T \Sigma$. 
\end{theorem} 

\begin{proof}
	We follow both \cite{cm40} and \cite[Appendix A]{ACS18} in order to derive the full Hessian of the area functional along such $u$. We equip $\Sigma$ with the pulled-back metric $u^\ast g$, which we still denote by $g$, making $u$ an isometric immersion. The $g$-inner product is denoted $\lan\cdot,\cdot\ra $ and the induced volume form is denoted $\ed \Sigma$. First given any local coordinates $\{x^i\}$ on $\Sigma$ we may define 
$$g_{ij}(t)=\lan u_i, u_j\ra \quad \text{where $u_i=\de_{x^i} u(x,t)$}$$
and 
$$G(t)=\sqrt{\det g^{-1}(0)g(t)}$$
which we note is well-defined independently of the choice of local coordinates. 

The area formula gives 
$$\A^m(u(\cdot,t))=\int_\Sigma G(t)\id \Sigma.$$

We note the following useful facts for the proceeding computations, where now $\D$ denotes the pulled back Levi-Civita connection by $u(x,t)$: 
 The $t$ and $x^i$ derivatives commute giving (here and throughout we use the short-hand $u_{tt}= \nabla_{u_t} u_t$),
 $$\de_t g_{ij}(t)=\lan \D_{u_i} u_t, u_j\ra + \lan u_i, \D_{u_j} u_t\ra \quad\text{and}\quad  
 \lan\D_{u_t}\D_{u_i} u_t, u_j\ra=\lan \D_{u_i} u_{tt}, u_j\ra - Rm^N(u_t,u_i,u_j,u_t),$$
 which together imply that, taking a Taylor expansion about $t=0$,
 $$g_{ij}(t) = g_{ij}(0) + tA +  t^2 B + O(t^3)$$
where 
$$A = \de_t g_{ij}(0) = \lan \D_{u_i} v, u_j\ra + \lan u_i, \D_{u_j} v\ra$$
and
$$B = \fr12\de_{tt}g_{ij}(0)= \lan \D_{u_i} v, \D_{u_j} v\ra + \fr12\lan \D_{u_i} u_{tt}, u_j\ra  + \fr12\lan u_i, \D_{u_j} u_{tt}\ra - Rm^N(v,u_i,u_j,v).$$
 
In particular, choosing coordinates $\{x^i\}$ so that $\{u_i= E_i\}$ is an orthonormal basis at a point $p$ when $t=0$, and using that 
$$\sqrt{\det (I+tA +  t^2 B + O(t^3))} = 1+\fr{t}{2}\textrm{tr}(A) + \fr{t^2}{2}\left(\textrm{tr}(B) + \fr14(\textrm{tr}(A))^2 - \fr12 \textrm{tr}(A^2)\right) + O(t^3)$$
we have that, at the point $p$:
\begin{eqnarray*}
 G(t) &=& 1 + t\textrm{div}_{\Sigma}(v) + \\
 &&\fr{t^2}{2}\left(|(\D v)^{\bot}|^2 + \textrm{div}_{\Sigma}(u_{tt}) - \sum_{i}Rm^N(v,E_i,E_i,v)+ \right. \\
 &&
\qquad \qquad \qquad  \left.
 + (\textrm{div}_{\Sigma}(v))^2- \sum_{ij}\lan \D_{E_i} v,E_j\ra \lan \D_{E_i} v,E_j\ra\right) + O(t^3).	 \end{eqnarray*}
 As is standard, setting $\textrm{div}_{\Sigma}(w) = \sum_{i} \lan \D_{E_i} w, E_i\ra$ for any $w\in u^\ast \T N$ and orthonormal frame $\{E_i\}$ we have that $\textrm{div}_\Sigma(v)=\textrm{div}_\Sigma(\sigma) - \lan s, H\ra$ where now $\textrm{div}_\Sigma(\sigma)$ can be interpreted as an intrinsic divergence along $\Sigma$ itself. 
Clearly   \begin{eqnarray*}
 	\ed \A^m(u)[v] &=& \pl{}{t}\vlinesub{t=0}\A^m(u(\cdot,t)) = \int_{\Sigma} \textrm{div}_{\Sigma}(v) \id \Sigma = -\int_\Sigma \lan H, s\ra \id \Sigma .
 \end{eqnarray*}
 The full Hessian is therefore: 
 \begin{eqnarray*}
 	\ed^2 \A^m(u)[v,v] &=& \pl{^2}{t^2}\vlinesub{t=0}\A^m(u(\cdot,t)) - \ed \A^m(u)[u_{tt}(\cdot,0)] \\
 	&=& \int_\Sigma |(\D v)^{\bot}|^2  - \sum_{i}Rm^N(v,E_i,E_i,v) +(\textrm{div}_{\Sigma}(v))^2- \sum_{i,j}\lan \D_{E_i} v,E_j\ra \lan \D_{E_i} v,E_j\ra \id \Sigma \\
 	&=& \int_{\Sigma} |\D^\bot s|^2 - \sum_{i,j} \lan A(E_i,E_j),s\ra^2 - \sum_i Rm^N(s,E_i,E_i,s) + \lan s, H\ra^2 + \mathcal{B} + \mathcal{C} \id \Sigma 
 \end{eqnarray*}
 where 
 $$\mathcal{B}=\sum_i|A(E_i,\sigma)|^2  - Rm^N(\sigma,E_i,E_i,\sigma) + (\textrm{div}_{\Sigma}(\sigma))^2 - \sum_{i,j}\lan \D_{E_i} \sigma,E_j\ra \lan \D_{E_i} \sigma,E_j\ra$$
 and 
 $$\mathcal{C} = 2\sum_{i}\lan \D^\bot_{E_i} s, A(E_i,\sigma)\ra  - 2Rm^N(\sigma,E_i,E_i,v) - 2\textrm{div}_{\Sigma}(\sigma)\lan s, H\ra +2\sum_i \lan A(\D^\top_{E_i} \sigma , E_i), s\ra.$$
 
 Directly from \cite[Appendix A]{ACS18} we see that 
 $$\mathcal{C}=\textrm{div}_{\Sigma}(\om)  + 2\lan \Db_\sigma s, H\ra   \quad \text{and} \quad \mathcal{B} = \lan A(\sigma,\sigma), H\ra  + \textrm{div}_{\Sigma}(\textrm{div}_{\Sigma}(\sigma)\sigma - \D^\top_{\sigma} \sigma)$$
 where $\om$ is a one-form defined along $\Sigma$ by $\om(\xi) = 2\lan A(\sigma,\xi),s\ra - 2\lan s, H\ra \lan \sigma, \xi\ra $ and we arrive at the desired formula.  
\end{proof}

\subsection{Conformally invariant variational problems and enclosed volume} 

Let $(N^{n+1},g)$ be an arbitrary smooth Riemannian manifold equipped with a two-form $\al \in \Gamma(\wedge^2\T^\ast N)$. For $u:\Sigma \to N$ and $\Sigma$ a closed orientable surface, let 
\begin{equation}
V_\al(u) : =\int_\Sigma u^\ast \al.
\end{equation}

If $u_t$ is a one-parameter family of variations with $\de_t u_t \vline_{t=0} = v\in \Gamma (u^\ast \T N)$, we have (see e.g. \cite{dalio-gianocca-riviere-2023}) 
$$\ed V_\al (u)[v] = \pl{}{t}\vlinesub{t=0}V_\al(u_t) = \int_\Sigma \ed \al (v,u_x,u_y)\id x\ed y$$
where $\{x,y\}$ is a local choice of oriented coordinates and one may straightforwardly check that the above integral is independent of coordinates and globally well-defined.

The second variation has been derived in \cite{dalio-gianocca-riviere-2023} and may be written: 
$$\ed^2 V_\al (u)[v,v]:=\int_\Sigma \ed \al(v,v_x,u_y) + \ed \al (v,u_x,v_y) + \rbrak{\D_v \ed \al }(v,u_x,u_y)\id x\ed y.$$

If $N$ is an orientable three-manifold then $\ed \al = \h \ed V_{N}$ for some smooth function $\h$ on $N$. In this case, we may write $E_\h$ and $V_\h$ rather than  $E_\al$ and $V_\al$. In the case that $u$ is an immersion we have  
$$ \ed V_\al(u)[v] = \int_\Sigma \lan v, \h \nu \ra \id \Sigma $$
where $\nu$ is a unit vector such that $(u_x, u_y, \nu)$ is positively oriented in $N$ and $\ed \Sigma$ is the induced area on $\Sigma$. 

In particular if $u:\Sigma \to N$ satisfies $u^{-1}(B_\eps(x)) = \emptyset$ then defining $\al_h$ via \eqref{eq:alH} we have: $u$ is a CMC $h$-immersion if and only if  
$$\ed\A_h(u)[v] = - \int_\Sigma \lan v, H\ra \id \Sigma + \int_\Sigma \lan v, h\nu\ra \id \Sigma =0 \quad\text{for all $v\in \Gamma(u^\ast \T N)$.}$$
\begin{lemma}\label{lem}
Suppose that $N$ is an orientable three-manifold equipped with a two-form $\al \in \Gamma(\wedge^2\T^\ast N)$.
Then, writing $\dform\al = \h\dform V_{N}$, we have that for any smooth possibly branched immersion $u:\Sigma \to N$,
\begin{multline}
\secondvar{V_\al}{u}{v,v} = \int_\Si -\inner{s,\h\nu}\inner{s,H} - 2\inner{\Db_\sigma  s,\h\nu} - \inner{A(\sigma,\sigma),\h\nu} + \inner{s,\nu}\nabla_s\h \id\Si
\end{multline}
where we employ the usual decomposition $v=\sigma + s$ into the tangential and normal components, $\sigma$ and $s$ respectively, of $v$. 
\end{lemma}

\begin{corollary}
\label{cor:a3}
If $u:\Si\to N$ is a CMC $h$ immersion, i.e. $V_\al = V_h$ and $\h\equiv h$, then the mean curvature of $u$ is $H=h\nu$ and 
$$ \secondvar{V_h}{u}{v,v} = \int_\Si -\inner{s,H}^2 - 2\inner{\nabla^\perp_\sigma s,H} - \inner{A(\sigma,\sigma),H} \id\Si. $$
\end{corollary}
\begin{remark}
Combining this with Theorem \ref{thm:app1} we have, if $u$ is a CMC $h$-immersion,
\begin{equation}\label{eqn:second-var-area-h}
\secondvar{\A_h}{u}{v,v} = \int_\Si \abs{\nabla^\perp s}^2 - \rbrak{\abs{A}^2 + Ric(\nu,\nu)}\abs{s}^2 \id\Si.
\end{equation}	
\end{remark}
\begin{proof}[Proof of Lemma \ref{lem}]
Setting for convenience $\om=\ed V_N$ so that $\dform\al=\h\om$, we have from \cite{dalio-gianocca-riviere-2023} that 
$$ \secondvar{V_\al}{u}{v,v} = \int_\Si \h\om(v,v_x,u_y) + \h\om(v,u_x,v_y) + \nabla_v\h\,\om(v,u_x,u_y) \id x \ed y. $$
Let $\nu$ be a normal vector field such that $\om(\nu,u_x,u_u) = e^{2\la}$.
Then we have that $\om(v,u_x,u_y) = \om(s,u_x,u_y) = \inner{s,\nu}\om(\nu,u_x,u_y)=\inner{s,\nu}e^{2\la}$, so that the third term above is
$$ \nabla_v\h\,\om(v,u_x,u_y) = \inner{s,\nu}e^{2\la}\nabla_v\h. $$

Now, the first term in the above integral is
\begin{align}\label{eqn:second-var-vol-lemma-1}
\h\om(v,v_x,u_y) &= \h\om(\inner{v,\nu}\nu, v_x,u_y) + \h\om(\sigma,v_x,u_y) \nonumber\\
	&= \h\inner{s,\nu}\om(\nu,\inner{v_x,u_x}e^{-2\la}u_x,u_y) + \h\om(\sigma,\inner{v_x,u_x}e^{-2\la}u_x,u_y) \nonumber\\
	&= \h\inner{s,\nu}\inner{v_x,u_x} - \h\inner{v_x,\nu}\om(\nu,\sigma,u_y) \nonumber\\
	&= \h\inner{s,\nu}\inner{v_x,u_x} - \h\inner{v_x,\nu}\om(\nu,e^{-2\la}\inner{\sigma,u_x}u_x,u_y)\nonumber \\
	&= \h\inner{s,\nu}\inner{v_x,u_x} - \h\inner{v_x,\nu}\inner{\sigma,u_x}.
\end{align}
Similarly, the second term is
\begin{equation}\label{eqn:second-var-vol-lemma-2}
\h\om(v,u_x,v_y) = \h\inner{s,\nu}\inner{v_y,u_y} - \h\inner{v_y,\nu}\inner{\sigma,u_y}.
\end{equation}
Note now that
\begin{align*}
\textrm{div}_\Si (v) &= \sum_i\inner{\nabla_{e_i} v,e_i} \\
	&= e^{-2\la}\inner{\nabla_{u_x}v,u_x} + e^{-2\la}\inner{\nabla_{u_y} v, u_y} \\
	&= e^{-2\la}\rbrak{\inner{v_x,u_x} + \inner{v_y,u_y}}.
\end{align*}
Thus combining \eqref{eqn:second-var-vol-lemma-1} \eqref{eqn:second-var-vol-lemma-2}, one obtains
\begin{align}\label{eqn:second-var-vol-lemma-3}
\h\om(v,v_x,u_y) + \h\om(v,u_x,v_y) &= \h\inner{s,\nu}e^{2\la}\textrm{div}_\Si(v) - \h\inner{v_x,\nu}\inner{\sigma,u_x} - \h\inner{v_y,\nu}\inner{\sigma,u_y} \nonumber\\
	&= -\inner{s,\h\nu}e^{2\la}\textrm{div}_\Si(\sigma) - \h\inner{s,\nu}H\inner{s,\nu} \nonumber\\
	&\qquad - \h\inner{v_x,\nu}\inner{\sigma,u_x} - \h\inner{v_y,\nu}\inner{\sigma,u_y}
\end{align}
where $H$ is the mean curvature along the immersion $u:\Si\to N$.
We focus now on the last two terms; indeed, we have
\begin{align*}
-\h\inner{v_x,\nu}\inner{\sigma,u_x} - \h\inner{v_y,\nu}\inner{\sigma,u_y} &= -\h\inner{\sigma_x,\nu}\inner{\sigma,u_x} - \h\inner{s_x,\nu}\inner{\sigma,u_x} \\
	&\quad - \h\inner{\sigma_y,\nu}\inner{\sigma,u_y} - \h\inner{s_y,\nu}\inner{\sigma,u_y}. \\
\end{align*}
The first and third terms combine to give $-e^{2\la}\h\inner{\nabla_\sigma\sigma,\nu}$, and the second and fourth terms combine to give $-e^{2\la}\h\inner{\nabla^\perp_\sigma s,\nu}$; thus we have
\begin{equation}\label{eqn:second-var-vol-lemma-4}
-\h\inner{v_x,\nu}\inner{\sigma,u_x} - \h\inner{v_y,\nu}\inner{\sigma,u_y} = -e^{2\la}\inner{A(\sigma,\sigma),\h\nu} - e^{2\la}\inner{\nabla^\perp_\sigma s,\h\nu}.
\end{equation}
Combining \eqref{eqn:second-var-vol-lemma-3} and \eqref{eqn:second-var-vol-lemma-4} we have
\begin{multline}
\secondvar{V_\al}{u}{v,v} = \int_\Si\h\inner{s,\nu}\textrm{div}_\Si(\sigma) - \inner{s,\h\nu}\inner{s,H\nu} - \inner{A(\sigma,\sigma),\h\nu} \\
	- \inner{\nabla^\perp_\sigma s,\h\nu} + \nabla_\sigma\h \inner{s,\nu} + \nabla_s\h\inner{s,\nu} \id\Si. \nn
\end{multline}
Now, note that $\h\inner{s,\nu}\textrm{div}_\Si(\sigma) = \textrm{div}_\Si(\inner{s,\h\nu}\sigma) - \inner{\nabla^\perp_\sigma s,\h\nu} - \inner{s,\rbrak{\nabla_\sigma\h}\nu}$, and so
\begin{multline}
\secondvar{V_\al}{u}{v,v} = \int_\Si -\inner{s,\h\nu}\inner{s,H\nu} - 2\inner{\nabla^\perp_\sigma s,\h\nu} - \inner{A(\sigma,\sigma),\h\nu} + \inner{s,\nu}\nabla_s\h \dform\Si . \nn
\end{multline}
\end{proof}

%----------------BIBLIOGRAPHY-------------------------------------------

%\bibliography{mathrefs.bib}
\bibliographystyle{plain}

\vspace{.5cm}

\begin{flushleft}

L. Seemungal: \textit{l.seemungal@leeds.ac.uk} 

B. Sharp: \textit{b.g.sharp@leeds.ac.uk}  

School of Mathematics, University of Leeds, Leeds LS2 9JT, UK  
\end{flushleft}

\end{document}